\title
{Longer cycles in vertex transitive graphs}
\author{
	Matt DeVos	\\
	{\normalsize mdevos@sfu.ca}
}

\date{}

\documentclass[12pt]{article}

\usepackage{amsmath}
\usepackage{amsthm}
\usepackage{amsfonts}

\begin{document}
\bibliographystyle{plain}
\maketitle
\setcounter{page}{1}
\newtheorem{theorem}{Theorem}
\newtheorem{lemma}[theorem]{Lemma}
\newtheorem{corollary}[theorem]{Corollary}
\newtheorem{proposition}[theorem]{Proposition}
\newtheorem{definition}[theorem]{Definition}
\newtheorem{claim}{Claim}
\newtheorem{conjecture}[theorem]{Conjecture}
\newtheorem{observation}[theorem]{Observation}
\newtheorem{problem}[theorem]{Problem}
\newtheorem{question}[theorem]{Question}

\begin{abstract}
In 1979 Babai found a clever argument to prove that every connected vertex transitive graph on $n \ge 3$ vertices contains a cycle of length at least $\sqrt{3n}$.  Here we modify his approach to show that such graphs must contain a cycle of length at least $\big( 1 - o(1) \big) n^{3/5}$.
\end{abstract}

A classic problem area in algebraic graph theory concerns the presence of long cycles (or paths) in connected vertex transitive graphs.  Apart from graphs on fewer than 3 vertices, there are only four graphs in this class known to not have Hamiltonian cycles: the Petersen graph, the Coxeter graph, and their truncations\footnote{To truncate a cubic graph, for every vertex $v$ incident with $e_1, e_2, e_3$, add new vertices $v_1, v_2, v_3$ all pairwise adjacent, change $e_i$ to have $v_i$ has an end instead of $v$ for $1 \le i \le 3$ and then delete $v$.}.    In particular, it is possible that every connected vertex transitive graph has a Hamiltonian path (a question of Lov\'asz) and it is possible that every connected Cayley graph on at least 3 vertices has a Hamiltonian cycle (a folklore conjecture\footnote{It seems clear that this has been independently conjectured by numerous researchers.}).  There is a considerable body of work demonstrating that certain restricted classes of connected vertex transitive graphs are Hamiltonian, but there seems to have been little attention given to finding general lower bounds.  Indeed, the only such result we know of is the following theorem from 1979.  

\begin{theorem}[Babai \cite{babai}]
\label{thetheorem}
Every connected vertex transitive graph with $n \ge 3$ vertices contains a cycle of length at least $\sqrt{3n}$.
\end{theorem}

Let us note that in contrast to the above dialogue suggesting that connected vertex transitive graphs may be nearly Hamiltonian, Babai has offered a conjecture in the other direction.  He conjectures the existence of a constant $c > 0$ so that there exist arbitrarily large connected vertex transitive graphs for which the length of the longest cycle is at most $(1-c)$ times the number of vertices.

In this note we follow a similar approach to that used by Babai in order to improve his bound to $\big( 1 - o(1) \big) n^{3/5}$.  The key tool we need is the following double-counting lemma that is a variation of one Babai employs.  Here we let our groups act on the right, so if the group $G$ acts on the set $V$ we use $v^g$ to denote the image of a point $v \in V$ under the permutation associated with $g \in G$.  We let $G_v = \{ g \in G \mid v^g = v \}$ denote the \emph{stabilizer} of $v$. 

\begin{lemma}
\label{thelemma}
Let $G$ be a finite group acting transitively on the set $V$, let $B,C \subseteq V$ and let $k \ge 0$.  If $|B \cap C^g| \ge k$ holds for every $g \in G$, then $|B| |C| \ge k |V|$.
\end{lemma}

\begin{proof}
We count the members of the set $S = \{ (g,y) \in G \times V \mid y \in B \cap C^g \}$ in two ways.  First, observe that we have $|B| |C|$ ways to choose $y \in B$ and $x \in C$, and for each such choice there are exactly $|G_y|$ group elements $g$ satisfying $x^g = y$.  It follows from this that $|S| = |B| |C| |G_y|$.  On the other hand, the assumption $|B \cap C^g| \ge k$, implies that for all $|G|$ choices of $g$ there are at least $k$ valid choices for $y$, giving the bound $|S| \ge k |G|$.  Now we have $|B| |C| |G_y| = |S| \ge k |G|$, and the result follows from the Orbit Stabilizer Theorem ($|G| = |V| |G_y|$).  
\end{proof}

To see how this lemma may be applied, suppose that $B$ is a set of vertices in a vertex transitive graph and that $B$ hits every longest cycle in at least $k$ points.  Setting $C$ to be the vertex set of a longest cycle (and $G$ to be the automorphism group of the graph), the conditions in the lemma are satisfied, so the longest cycle must have length at least $\frac{k |V| }{|B|}$.  In short, this lemma gives us a lower bound on the length of a longest cycle from the existence of an efficient hitting set for longest cycles.

For Babai's argument, he takes his hitting set $B$ to be the vertex set of a longest cycle.  Connected vertex transitive graphs on $\ge 3$ vertices are either cycles (for which we have nothing to prove) or are 3-connected by a theorem of Mader and Watkins (see \cite{agt}).  An easy exercise shows that in a 3-connected graph any two longest cycles must intersect in at least 3 vertices.  Thus, taking $B = C$ to be the vertex set of a longest cycle and applying the  preceding lemma with $k=3$ gives us Babai's Theorem: $|C| \ge \sqrt{3 |V|}$.

Our argument uses this idea and the following additional lemma.

\begin{lemma}
Let $X$ be a 2-connected graph and let $C_1, C_2$ be longest cycles in $X$.  If $|V(C_1) \cap V(C_2)| = k$, then there exists a set of at most $k^2 + k$ vertices hitting all longest cycles.
\end{lemma}

\begin{proof}
If the graph $X' = X - \big( V(C_1) \cap V(C_2) \big)$ contains more than $k^2$ vertex disjoint paths from $V(C_1) \setminus V(C_2)$ to $V(C_2) \setminus V(C_1)$, then this graph contains two vertex disjoint paths, say $P$ and $P'$, internally disjoint from $V(C_1) \cup V(C_2)$ both having initial vertex in the same component, say $H_1$, of the subgraph $C_1 \setminus V(C_2)$, and both having terminal vertex in the same component, say $H_2$, of the subgraph $C_2 \setminus V(C_1)$.  However, this gives a contradiction since we can now find two cycles $C_1', C_2'$ with $C_i' \subseteq C_i \cup P \cup P' \cup H_{3-i}$ satisfying $|V(C_1')| + |V(C_2')| = |V(C_1)| + |V(C_2)| + 2| V(P)| + 2 |V(P')|$.  Therefore, such paths do not exist and we may apply Menger's Theorem to choose a set $B_0$ of at most $k^2$ vertices separating $V(C_1) \setminus V(C_2)$ and $V(C_2) \setminus V(C_1)$ in $X'$.  Now define $B = B_0 \cup \big( V(C_1) \cap V(C_2) \big)$.  Note that $|B| \le k^2 + k$ and that $B$ separates $V(C_1)$ and $V(C_2)$ in the original graph $X$.  Since $X$ is 2-connected, any two longest cycles in $X$ must intersect.  In particular, every longest cycle in $X$ must contain a vertex in both $V(C_1)$ and in $V(C_2)$ and must therefore intersect~$B$.  
\end{proof}

With this, we can prove our main result.

\begin{theorem}
Every connected vertex transitive graph on $n \ge 3$ vertices contains a cycle of length at least $(1 - o(1)) n^{3/5}$.  
\end{theorem}

\begin{proof}
Let $X$ be a connected vertex transitive graph with $n \ge 3$ vertices.  Note that $X$ must be 2-connected by vertex transitivity and the observation that a leaf of a spanning tree is not a cut-vertex.  Let $t$ be the length of a longest cycle and let $k$ be the minimum size of $V(C_1) \cap V(C_2)$ over all longest cycles $C_1$ and $C_2$.  By applying Lemma \ref{thelemma} with $B$ and $C$ equal to the vertex set of a longest cycle we get the bound $t \ge \sqrt{ k n }$.  The previous lemma implies the existence of a set of size $k^2 +k$ that hits every longest cycle in at least one vertex, so another application of Lemma \ref{thelemma} gives us $t \ge \frac{n}{k^2+k}$.  The combined bound $t \ge \max\{ \sqrt{kn}, \frac{n}{k^2+k} \}$ is weakest when these bounds coincide and $n = k^5 + 2k^4 + k^3$ and this gives $t \ge  \big( 1 - o(1) \big) n^{3/5}$ as claimed.
\end{proof}

Let us comment that we know how to improve the constant $1 - o(1)$, but we do not know how to increase the exponent above ${3/5}$.

\end{document}